\documentclass[11pt,final]{amsart}
\usepackage{mathtools}

\usepackage{pgf,tikz}
\usetikzlibrary{arrows}
\definecolor{ffffqq}{rgb}{1,1,0}
\definecolor{qqffqq}{rgb}{0,1,0}
\definecolor{ttttff}{rgb}{0.2,0.2,1}
\definecolor{ffqqtt}{rgb}{1,0,0.2}
\usepackage[active]{srcltx}
\usepackage[latin1]{inputenc}
\usepackage{amsmath,bbm}
\usepackage{amsfonts}
\usepackage{pgf,tikz}
\usepackage{amsmath,accents,eqnarray}
\usetikzlibrary{arrows}
\usepackage{amssymb}
\usepackage[colorlinks=true,urlcolor=blue,
citecolor=red,linkcolor=blue,linktocpage,pdfpagelabels,bookmarksnumbered,bookmarksopen]{hyperref}
\usepackage{graphicx, tikz}
\usepackage{pgf,tikz}
\usetikzlibrary{arrows}
\definecolor{wwqqww}{rgb}{0.4,0,0,4}
\definecolor{wwffqq}{rgb}{0.4,1,0}
\definecolor{fffftt}{rgb}{1,1,0.2}
\definecolor{ccqqcc}{rgb}{0.8,0,0.8}
\definecolor{ffffww}{rgb}{1,1,0.4}
\definecolor{ttfftt}{rgb}{0.2,1,0.2}
\definecolor{fffftt}{rgb}{1,1,0.2}
\definecolor{ccqqww}{rgb}{0.8,0,0.4}
\definecolor{ffffww}{rgb}{1,1,0.4}
\definecolor{ttfftt}{rgb}{0.2,1,0.2}
\definecolor{uququq}{rgb}{0.25,0.25,0.25}
\definecolor{ffqqtt}{rgb}{1,0,0.2}
\definecolor{ttffff}{rgb}{0.2,1,1}
\definecolor{fffftt}{rgb}{1,1,0.2}
\definecolor{wwffqq}{rgb}{0.4,1,0}
\definecolor{qqqqff}{rgb}{0,0,1}
\definecolor{uququq}{rgb}{0.25,0.25,0.25}
\definecolor{ffqqtt}{rgb}{1,0,0.2}
\definecolor{fffftt}{rgb}{1,1,0.2}
\definecolor{wwffqq}{rgb}{0.4,1,0}
\definecolor{qqqqff}{rgb}{0,0,1}
\definecolor{ttfftt}{rgb}{0.2,1,0.2}
\definecolor{qqqqff}{rgb}{0,0,1}
\definecolor{ffqqtt}{rgb}{1,0,0.2}
\definecolor{fffftt}{rgb}{1,1,0.2}
\definecolor{ffffqq}{rgb}{1,1,0}
\definecolor{qqfftt}{rgb}{0,1,0.2}
\tikzset { domaine/.style 2 args={domain=#1:#2} }
\tikzset{
xmin/.store in=\xmin, xmin/.default=-3, xmin=-3,
xmax/.store in=\xmax, xmax/.default=3, xmax=3,
ymin/.store in=\ymin, ymin/.default=-3, ymin=-3,
ymax/.store in=\ymax, ymax/.default=3, ymax=3,
}

\usepackage{pgf,tikz}
\usetikzlibrary{arrows}
\definecolor{ffffqq}{rgb}{1,1,0}
\definecolor{ttfftt}{rgb}{0.2,1,0.2}
\definecolor{qqffcc}{rgb}{0,1,0.8}
\definecolor{ffwwqq}{rgb}{1,0.4,0}
\definecolor{ffqqqq}{rgb}{1,0,0}
\definecolor{ttffqq}{rgb}{0.2,1,0}
\definecolor{qqqqff}{rgb}{0,0,1}
\definecolor{ffqqtt}{rgb}{1,0,0.2}
\definecolor{ffffqq}{rgb}{1,1,0}
\definecolor{qqfftt}{rgb}{0,1,0.2}
\usepackage[english]{babel}
\usepackage{marginnote}
\usepackage[left=2.95cm,right=2.95cm,top=2.8cm,bottom=2.8cm]{geometry}
\numberwithin{equation}{section}
\usepackage{mathrsfs}
\usepackage{color}
\usepackage{amsmath, bbm}
\usepackage{amsfonts}
\usepackage{amssymb}
\usepackage{graphicx}%
\providecommand{\U}[1]{\protect\rule{.1in}{.1in}}
\definecolor{linkcolor}{rgb}{0.00,0.50,0.00}
\providecommand{\U}[1]{\protect\rule{.1in}{.1in}}
\newtheorem{theorem}{Theorem}[section]
\newtheorem{proposition}[theorem]{Proposition}
\newtheorem{lemma}[theorem]{Lemma}

\theoremstyle{remark}
\newtheorem{remark}[theorem]{Remark}

\numberwithin{equation}{section}

\newcommand{\sigmatau}{\sigma^{(\tau)}}

\newcommand{\dd}{\mathrm{d}}

\newcommand{\res}{\mathop{\hbox{\vrule height 7pt width .5pt depth 0pt \vrule height .5pt width 6pt depth 0pt}}\nolimits}
\newcommand{\deb}{\rightharpoonup}

\newcommand{\surf}{\mathcal{H}^{d-1}}

\newcommand{\lcal}{\mathcal L}

\newcommand{\impl}{\Rightarrow}
\newcommand{\ve}{\varepsilon}

\newcommand{\R}{\mathbb R}

\newcommand{\ind}{\mathbbm 1}

\DeclareMathOperator{\spt}{spt}
\DeclareMathOperator{\Lip}{Lip}

\DeclareMathOperator{\diam}{diam}

\usepackage{pgf,tikz}
\usetikzlibrary{arrows}
\definecolor{zzttqq}{rgb}{0.6,0.2,0}
\definecolor{uququq}{rgb}{0.25,0.25,0.25}
\definecolor{qqqqff}{rgb}{0,0,1}
\definecolor{xdxdff}{rgb}{0.49,0.49,1}
\usepackage{pgf,tikz}
\usetikzlibrary{arrows}
\definecolor{ffffqq}{rgb}{1,1,0}
\definecolor{fftttt}{rgb}{1,0.2,0.2}
\definecolor{ffqqqq}{rgb}{1,0,0}
\definecolor{ffqqtt}{rgb}{1,0,0.2}
\definecolor{ffttqq}{rgb}{1,0.2,0}
\definecolor{xdxdff}{rgb}{0.49,0.49,1}
\definecolor{qqqqff}{rgb}{0,0,1}
\definecolor{ffttww}{rgb}{1,0.2,0.4}
\definecolor{ttqqcc}{rgb}{0.2,0,0.8}
\definecolor{ffqqww}{rgb}{1,0,0.4}
\definecolor{xdxdff}{rgb}{0.49,0.49,1}
\definecolor{qqqqff}{rgb}{0,0,1}
\definecolor{uququq}{rgb}{0.25,0.25,0.25}
\definecolor{ttffww}{rgb}{0.2,1,0.4}
\definecolor{ffqqqq}{rgb}{1,0,0}
\definecolor{qqfftt}{rgb}{0,1,0.2}
\definecolor{zzttqq}{rgb}{0.6,0.2,0}
\definecolor{ffffqq}{rgb}{1,1,0}
\definecolor{qqffww}{rgb}{0,1,0.4}
\definecolor{xdxdff}{rgb}{0.49,0.49,1}
\definecolor{qqqqff}{rgb}{0,0,1}

\definecolor{ffqqqq}{rgb}{1,0,0}
\definecolor{ttqqff}{rgb}{0.2,0,1}
\definecolor{ttqqcc}{rgb}{0.2,0,0.8}
\definecolor{uququq}{rgb}{0.25,0.25,0.25}
\definecolor{xdxdff}{rgb}{0.49,0.49,1}
\definecolor{qqqqff}{rgb}{0,0,1}
\definecolor{ffqqtt}{rgb}{1,0,0.2}

\definecolor{qqffqq}{rgb}{0,1,0}
\definecolor{qqfftt}{rgb}{0,1,0.2}
\definecolor{ffffqq}{rgb}{1,1,0}
\definecolor{ffttcc}{rgb}{1,0.2,0.8}
\definecolor{ffwwqq}{rgb}{1,0.4,0}
\definecolor{ffqqqq}{rgb}{1,0,0}
\definecolor{ffffqq}{rgb}{1,1,0}
\definecolor{qqfftt}{rgb}{0,1,0.2}
\definecolor{ffttcc}{rgb}{1,0.2,0.8}
\definecolor{ffwwqq}{rgb}{1,0.4,0}
\definecolor{ffqqqq}{rgb}{1,0,0}
\usepackage{subfig}
\usetikzlibrary{arrows}
\definecolor{ffqqtt}{rgb}{1,0,0.2}
\definecolor{ffffqq}{rgb}{1,1,0}
\definecolor{ttffqq}{rgb}{0.2,1,0}
\definecolor{wwwwww}{rgb}{0.4,0.4,0.4}
\definecolor{zzqqzz}{rgb}{0.6,0,0,6}
\begin{document}
\title[BV least gradient problem and transport densities]
{$L^p$ bounds for boundary-to-boundary transport densities, and $W^{1,p}$ bounds for the BV least gradient problem in 2D
}
\author[S. Dweik, F. Santambrogio]{Samer Dweik and Filippo Santambrogio}
\address{Laboratoire de Math\'ematiques d'Orsay, Univ. Paris-Sud, CNRS, Universit\'e Paris-Saclay, 91405 Orsay Cedex, France}
\email{samer.dweik@math.u-psud.fr, filippo.santambrogio@math.u-psud.fr}
\maketitle
\begin{abstract}
The least gradient problem (minimizing the total variation with given boundary data) is equivalent, in the plane, to the Beckmann minimal-flow problem with source and target measures located on the boundary of the domain, which is in turn related to an optimal transport problem. Motivated by this fact, we prove $L^p$ summability results for the solution of the Beckmann problem in this setting, which improve upon previous results where the measures were themselves supposed to be \,$L^p$. In the plane, we carry out all the analysis for general strictly convex norms, which requires to first introduce the corresponding optimal transport tools. We then obtain results about the $W^{1,p}$ regularity of the solution of the anisotropic least gradient problem in uniformly convex domains.
\end{abstract}
\section{Introduction} \label{intro}
A classical problem in calculus of variations, which is of interest both with applications in image processing but also for its connection with minimal surfaces, is the so-called least gradient problem, considered for instance in \cite{Bom,55,Rybka,22,33,44}. This is the problem of minimizing the total variation of the vector measure \,$\nabla u$\, among all BV functions $u$ defined on an open domain \,$\Omega$\, with given boundary datum. To be more general, we will directly consider the anisotropic case (\cite{Mercier,Mazonaniso}), using an arbitrary strictly convex norm $\varphi$ in \,$\mathbb{R}^d$. Consider\\
\begin{equation} \label{leastgradient-inf}
\inf\bigg\{\int_\Omega \varphi(\nabla u)\,:\,u \in BV(\Omega),\,u_{|\partial\Omega}=g\bigg\},
\end{equation}\\
where $u_{|\partial\Omega}$\, denotes the trace of \,$u$ in the sense of BV functions and $\int_\Omega \varphi(\nabla u)$ is the $\varphi-$total variation measure of $\nabla u$ (i.e., $\int_\Omega \varphi(\nabla u):=\int_\Omega \varphi(\frac{d \nabla u}{d |\nabla u|}(x))\,\mathrm{d}|\nabla u|$), this problem relaxes into\\
\begin{equation} \label{leastgradient-pen}
\min\bigg\{\int_\Omega \varphi(\nabla u)+\int_{\partial\Omega}|u_{|\partial\Omega}-g|\,\dd\surf\,:\,u \in BV(\Omega)\bigg\}.
\end{equation}\\
This can also be expressed in the following way: extend $g$ into a BV function $\tilde g$ defined on a larger domain $\Omega'$, and then consider\\
\begin{equation} \label{leastgradient-ext}
\min\bigg\{\int_{\overline\Omega}\varphi(\nabla u)\,:\,u \in BV(\Omega'), \; u=\tilde g \,\mbox{ on }\, \Omega'\setminus \Omega \bigg\}.
\end{equation}\\
 The boundary datum $g$ should be taken as a possible trace of BV functions, i.e. in $L^1(\partial\Omega)$, yet, the fact that the (a) solution $u$ to \eqref{leastgradient-pen} and \eqref{leastgradient-ext} satisfies or not $u_{|\partial\Omega}=g$ could depend on $g$ (and on the domain). In case we have $u_{|\partial\Omega}=g$, then $u$ is also a solution of  \eqref{leastgradient-inf}. 
 In the Euclidean case (i.e., when $\varphi=|\cdot|$), the author of \cite{55} proves existence of solutions to \eqref{leastgradient-inf} for boundary data in $BV(\partial\Omega)$, while, in \cite{ST}, the authors give an example of a function \,$g$\, such that \eqref{leastgradient-inf} has no solution ($g$ was chosen to be the characteristic function of a certain fat Cantor set, which does not lie in $BV(\partial\Omega)$).
 
In this paper the anisotropic least gradient problem will only be considered in the planar case $\Omega \subset \mathbb{R}^2$, and the boundary datum \,$g$\, will be at least in $BV(\partial\Omega)$ (something which makes perfectly sense, since $\partial \Omega$ is a closed curve, and we are just speaking about BV functions in 1D). Following \cite{Rybka}, we can see that there is a one-to-one correspondence between vector measures $\nabla u$ in \eqref{leastgradient-ext}  (considered as measures on $\overline\Omega$, so that we also include the part of the derivative of $u$ which is on the boundary, i.e. the possible jump from $u_{|\partial\Omega}$ to $g$) and vector measures $w$ satisfying, in $\overline\Omega$, $\nabla\cdot w=f$ where $f$ is the measure obtained as the tangential derivative of $g\in BV(\partial\Omega)$; moreover, the mass of $\nabla u$ and of $w$ are the same. Indeed, one just needs to take $w=R_{\frac{\pi}{2}} \nabla u$, where $R_{\theta}$ denotes a rotation with angle $\theta$ around the origin, and $w$ solves the following problem\\
\begin{equation} \label{Beckmann}
\inf\bigg\{ ||w||(\overline\Omega)\,:\,w \in \mathcal M^2(\overline\Omega),\,\nabla\cdot w=f\bigg\},
\end{equation}\\
where $\mathcal M^d(\overline\Omega)$ is the space of finite vector measures on $\overline\Omega$ valued in $\R^d$ (here, $d=2$), $|| \cdot ||$ is the {\it rotation-norm} of $\varphi$ (i.e. $||v||:=\varphi(R_{-\frac{\pi}{2}} v)$ for every $v \in \mathbb{R}^2$) and, $||w||$ denotes the variation measure associated with the vector measure $w$, i.e., $||w||(E):=\sup\left\{\sum_{i} ||w(A_i)||\right\}$, for every measurable set $E \subset \Omega$, where the supremum is taken over all partitions $E=\bigcup_i A_i$\, into a countable number of disjoint measurable subsets.
If we identify $f=\partial g / \partial \mathbf{t}$ ($\mathbf{t}:=R_{-\frac{\pi}{2}}\mathbf{n}$ standing for the tangent vector to $\partial\Omega$) with its restriction to the boundary, we can also write the condition $\nabla\cdot w=f$ as $\nabla\cdot w=0$  in $\Omega,\,w \cdot \mathbf{n}=f $ on $\partial\Omega$ (in general, when we write $\nabla\cdot w=f$ we mean $\int \nabla\phi\cdot \dd w=-\int \phi\,\dd f$ for every smooth test function $\phi$, without imposing $\phi$ to have compact support, i.e. we also include boundary conditions).\\

The study of the anisotropic least gradient problem can consequently be done by studying \eqref{Beckmann}, and the question whether \eqref{leastgradient-inf} has a solution becomes whether the solution to \eqref{Beckmann} gives mass to the boundary or not.\\

The important point is that  the Beckmann problem is strongly related to optimal transport theory, and is in some sense equivalent to the Monge problem \\
\begin{equation}   \label{Kanto}
  \min\left\{\int_{ \overline{\Omega} \times \overline{\Omega}}||x-y||\,\mathrm{d}\gamma:\;
  \gamma \in \mathcal{M}^+(\overline{\Omega} \times \overline{\Omega}),\,(\Pi_x)_{\#} \gamma=f^+\,\,\mbox{ and}\,\,\,(\Pi_y)_{\#} \gamma=f^- \right\},
  \end{equation}\\
  where $f^\pm$ represent the positive and negative parts of $f$, i.e. two positive measures with the same mass
(see for instance \cite[Chapter 4]{8} for details about this equivalence in the Euclidean case, i.e. when $||\cdot||=|\cdot|$, or Section \ref{Sec.2} to the case of a general strictly convex norm $||\cdot||$). The scalar measure $\sigma=||w||$ obtained from an optimal $w$ is called {\it transport density} (see Section \ref{Sec.2}).

Because of its many connections to shape optimization problems (\cite{BouBut JEMS, BouButSep}), traffic congestion (\cite{BraCarSan}), image processing (\cite{LelLorSchVal}),\,\dots, many results are available in the literature about the transport density, and in particular its summability (most of these results concern only the Euclidean case). Of course, $L^p$ summability of \,$\sigma$\, is equivalent to \,$W^{1,p}$\, regularity of the optimal $u$. 

In the Euclidean case, $L^p$ summability results on the transport density have been analyzed in \cite{1, 2, 3, 7}: in dimension $d$,  for $p<d/(d-1)$, we have $w \in L^p$ as soon as at least one between $f^+$ or $f^-$ is in $L^p$, while for general $p$ (including $p=\infty$), this is proven when both \,$f^+$ and \,$f^-$ belong to $L^p(\Omega)$. But in the case of interest for applications to the least gradient problem, the measures $f^\pm$ are singular (they are concentrated on the negligible set $\partial\Omega$).
 The only result obtained so far with measures concentrated on the boundary is the one that we presented in \cite{5}, where we considered the case where $f^-$ is the projection of $f^+$ on $\partial\Omega$. On the other hand, this is far from the setting that we want to study now, since, first, in \cite{5}, only one of the two measures is on $\partial\Omega$, and second, it is not an arbitrary measure but it is chosen to be the projection of the other. 

We can say that, so far, the $L^p$ summability of \,$w$\, in the case where both the measures \,$f^+$ and \,$f^-$ are concentrated on the boundary is unknown (a counterexample is presented in a particular case in \cite{5}). In particular,
we do not know whether the optimal flow $w$ belongs or not to $L^p(\Omega)$ provided \,$f^\pm \in L^p(\partial\Omega)$. The goal of the present paper is exactly to investigate this kind of $L^p$ summability results under suitable assumptions on the domain $\Omega$, and then applying them to the $W^{1,p}$ regularity of the solution of \eqref{leastgradient-inf}. 

The paper is organized as follows. In Section \ref{Sec.2}, we adapt some well-known facts concerning the Monge-Kantorovich problem to the precise setting of transport from the boundary to the boundary.
In Section \ref{Sec.3}, we show positive results on the $L^p$ summability of $\sigma$ in the transport from a measure on the boundary to a measure on the boundary of a uniformly convex domain, for \,$d=2$\, in the case of a general strictly convex norm $||\cdot||$, or in arbitrary dimension \,$d\geq 2$\, with \,$||\cdot||=|\cdot|$. In particular, we see that $f^\pm\in L^p(\partial\Omega)\impl \sigma\in L^p(\Omega)$ holds for $p\leq 2$ (to go beyond $L^2$ summability one needs extra regularity of the data) in the case \,$d=2$\, with a general norm $||\cdot||$, or for every \,$d\geq 2$\, with \,$||\cdot||=|\cdot|$. Section \ref{Sec.4} gives indeed a counter-example where $f\in L^\infty$ but $\sigma\notin L^p$ for any $p>2$. Finally, Section \ref{Sec.5} summarizes the applications, in the case $d=2$, of these results to the anisotropic least gradient problem. 

Many of the results that we recover were already known, at least for the case of the Euclidean norm, thanks to different methods, but we believe that the connection with optimal transport and the technique we develop are interesting in themselves. Moreover, the generalizations to the anisotropic case are non-trivial via standard methods, while they are essentially straightforward via the present approach. As a last interesting point, to the best of our knowledge, the following statement is novel even in the Euclidean case: if $\Omega$ is a uniformly convex domain in dimension $2$, $p\leq 2$ and $g\in W^{1,p}(\partial\Omega)$, then the solution to \eqref{leastgradient-inf} exists, is unique, and belongs to $W^{1,p}(\Omega)$ (this is our Theorem \ref{onlynew}).

\section{Monge-Kantorovich and Beckmann problems} \label{Sec.2}

Let $||\cdot||$ be an arbitrary strictly convex norm in $\mathbb{R}^d$. Given two  finite positive Borel measures $f^+$ and $f^-$ on a compact convex domain $\Omega\subset\R^d$ (the closure of a non-empty convex open set, note that from now on, to make notation lighter, we will call $\Omega$ the closed domain, and not the open one), satisfying the mass balance condition $f^+(\Omega)=f^-(\Omega)$, we consider
the following minimization problem
\begin{equation}\label{kant|x-y|} \min\left\{\int_{ \Omega \times \Omega}||x-y||\,\mathrm{d}\gamma\,:\,\gamma \in \Pi(f^+,f^-)\right\},
  \end{equation}
  where
  $$\Pi(f^+,f^-):=\left\{\gamma \in \mathcal{M}^+( \Omega \times \Omega):\;(\Pi_{x})_{ \#}\gamma =f^+\;,\;(\Pi_{y})_{ \#}\gamma =f^-\;\right\}.$$\\
This is a relaxation of the classical Monge optimal transportation problem \cite{Monge}, which is the following
$$\inf\left\{\int_{\Omega}||x-T(x)||\,\mathrm{d}f^+:\;T_\#f^+=f^-\right\}.$$ \\
Actually, these two problems are equivalent as soon as one can prove that there exists an optimal $\gamma$ in \eqref{kant|x-y|}  which is concentrated on the graph of a measurable map $T$, i.e. $\gamma=(id,T)_\#f^+$. This is the case whenever $f^+ \ll \mathcal{L}^d$: the existence of an optimal map $T$ in this problem (or the fact that an optimal $\gamma$ is of the form $(id,T)_\#f^+$) has been a matter of active study between the end of the '90s and the beginning of this century, and we cite in particular \cite{ambro,ambros,Evans, TW,Caf} for the case of the Euclidean norm. For different norms, see \cite{Car,Cham-str} when the norm is strictly convex, and 
\cite{Cham} for the general case. However, since we are interested in transport problems where $f^+$ is concentrated on the negligible set $\partial\Omega$, we will discuss later a specific technique in this particular case.

We underline that, by a suitable inf-sup exchange procedure, it is possible to see that the following maximization problem
\begin{equation} \label{dual}
\max\left\{\int_{\Omega}\phi\,\mathrm{d}(f^+-f^-):\;||\nabla \phi||_{\star,\infty} \leq 1\right\},
\end{equation} \\ 
 is the dual of \,(\ref{kant|x-y|}) (its value equals $\min\,(\ref{kant|x-y|})$), where \\ 
 $$||\nabla \phi||_{\star,\infty}=\sup_{x \in \Omega} ||\nabla \phi (x)||_{\star}$$
and $$||v||_{\star}:=\sup\bigg\{v \cdot \xi  \,:\,||\xi|| \leq 1\bigg\},\,\,\,\mbox{for every}\,\,\,v \in \mathbb{R}^d.$$\\
 Note that, as soon as the domain \,$\Omega$\, is convex, the condition $||\nabla \phi||_{\star,\infty} \leq 1$ is nothing but the fact that $\phi$\, is $1-$Lip with respect to $||\cdot||$.  Moreover, the equality of the two optimal values implies that optimal \,$\gamma$\, and \,$\phi$\, satisfy \,$\phi(x)-\phi(y)=||x-y||$\, on the support of \,$\gamma$ but also that, whenever we find some admissible $\gamma$\, and \,$\phi$\, satisfying \,$\int_{\Omega \times \Omega}||x-y||\,\mathrm{d}\gamma=\int_{\Omega}\phi\,\mathrm{d}(f^+-f^-)$, they are both optimal. The maximizers in (\ref{dual}) are called {\it Kantorovich potentials}. Given a maximizer $\phi$, we call $\it{transport \,ray}$ any maximal segment $[x,y]$ such that \,$\phi(x) - \phi(y)=||x-y||$, and the important fact is that whenever a point $(x_0,y_0)$ belongs to the support of an optimal $\gamma$, then $x_0$ and $y_0$ must belong to a common transport ray. In addition, two different transport rays cannot intersect at an interior point of either of them (this strongly uses the strict convexity of the norm, see, for instance, \cite{8,Caf,Cham}). 
 
 In optimal transport theory it is classical to associate with any optimal transport plan \,$\gamma$\, a positive measure \,$\sigma_\gamma$\, on $\Omega$, called $\textit{transport density}$, which represents the amount of transport taking place in each region of $\Omega$. This measure \,$\sigma_\gamma$\, is defined via (this is in fact an adaptation of the definition of the transport density, see for instance \cite{8}, which is usually given in the Euclidean case) 
\begin{equation}\label{transport density def}
<\sigma_\gamma,\phi>=\int_{\Omega \times \Omega}\mathrm{d}\gamma(x,y)\int_{0}^{1}\phi(\omega_{x,y}(t))\,||{\omega}^\prime_{x,y}(t)||\,\mathrm{d}t\;\;\; \mbox{for all}\;\;\phi\;\in\;C(\Omega)
\end{equation}
where $\omega_{x,y}$ is a parameterization of the geodesic
connecting $x$ to $y$ (which is, thanks to the strict convexity of the norm $||\cdot||$, the straight line segment between \,$x$\, and \,$y$, and we take \,$\omega_{x,y}(t)=(1-t)x+ty$). 
Notice in particular that one can write, when the norm $||\cdot||$ is the Euclidean norm $|\cdot|$,
\begin{equation}\label{representationH1}
\sigma_\gamma(A)=\int_{\Omega \times \Omega} \mathcal H^1([x,y]\cap A)\,\mathrm{d}\gamma(x,y)\;\;\;\mbox{for every Borel set}\;A
\end{equation}
and similar formulas exist for other norms if one computes the $ \mathcal H^1$ measure with respect to the induced distance. 
This means that, for a subregion $A$, $\sigma_\gamma(A)$ stands for ``how much'' the transport takes place in $A$, if particles move from their origin $x$ to their destination $y$ on straight lines. We can also define a vectorial version of \,$\sigma_\gamma$, as a vector measure \,$w_\gamma$ on \,$\Omega$\, defined by
\begin{equation}\label{transport density def-v}
<w_\gamma,\xi>=\int_{\Omega \times \Omega}\mathrm{d}\gamma(x,y)\int_{0}^{1}\xi(\omega_{x,y}(t))\cdot {\omega}^\prime_{x,y}(t)\,\mathrm{d}t\;\;\; \mbox{for all}\;\;\xi\;\in\;C(\Omega,\R^d).
\end{equation}
One can show that this vector measure $w_\gamma$ solves the following problem (which is the so-called {\it continuous transportation model}\, proposed by Beckmann in \cite{0}):
\begin{equation} \label{Beckmann'}
\min\bigg\{ ||w||(\overline\Omega)\,:\,w \in \mathcal M^d(\overline\Omega),\,\nabla\cdot w=f\bigg\}.
\end{equation} 
Indeed, we can see easily that $w_\gamma$ is admissible in \eqref{Beckmann'} by using a gradient test function $\xi$. Moreover, we have $||w_\gamma|| \leq \sigma_\gamma,$ which implies $||w_\gamma||(\Omega) \leq \sigma_\gamma(\Omega) =\min\eqref{kant|x-y|}=\sup\eqref{dual} \leq \min\eqref{Beckmann'}$, where the last inequality follows from the fact that, if \,$\nabla \cdot w= f$\, and \,$||\nabla \phi||_{\star,\infty} \leq 1$, then we have $\int_\Omega \phi\,\mathrm{d}f = - \int_\Omega \nabla \phi \cdot \mathrm{d}w \leq ||w||(\Omega)$. This yields that \,$w_\gamma$ is an optimal flow for \eqref{Beckmann'}.

As we said, most of the analysis of the Beckmann problem has been performed so far in the Euclidean case, and we will summarize here the main achievements. In this case, it is well-known that every solution of the problem \eqref{Beckmann'} is of the form $w=w_\gamma$ for an optimal transport plan $\gamma$ for \eqref{kant|x-y|} (see, for instance, \cite[Chapter 4]{8}). As in general Problem \eqref{kant|x-y|} can admit several different solutions, also \eqref{Beckmann'} can have non-unique solutions. Yet, it is possible to prove that when either $f^+$ or $f^-$ are absolutely continuous measures, then all different optimal transport plans $\gamma$ induce the same vector measure $w_\gamma$. The following result also includes summability estimates:
\begin{proposition}\label{prop transp dens}
 Suppose that $||\cdot||$ is the Euclidean norm and $f^+\ll\lcal^d$. Then, the transport density \,$\sigma_\gamma$
 is unique (i.e., it does not depend on the choice of the optimal transport plan $\gamma$) and $\sigma_\gamma \ll \mathcal{L}^d$. If \,$f^+ \in L^p(\Omega)$ with $p<d/(d-1)$, then \,$\sigma_\gamma \in L^p(\Omega)$. Moreover, for arbitrary $p\in[1,\infty]$, if both $ f^+,\,f^- \in L^p(\Omega)$, then \,$\sigma_\gamma$\, also belongs to $L^p(\Omega)$.  
\end{proposition}
\noindent These properties are well-known in the literature, and we refer to \cite{5}, \cite{6},  \cite{7}, \cite{3} and \cite{7}, as well as to \cite[Chapter 4]{8}.

 We want now to prove that, even in the case of a general strictly convex norm $||\cdot||$, any optimal flow for \eqref{Beckmann'} comes from an optimal transport plan for \eqref{kant|x-y|}. First of all, we will introduce some objects that generalize both \,$\sigma_\gamma$\, and \,$w_\gamma$.
Let \,$\mathcal{C}$\, be the set of absolutely continuous curves \,$\omega: [0,1] \mapsto \Omega$. We call $\textit{traffic plan}$ any positive measure \,$Q$\, on \,$\mathcal{C}$\, with total mass equal to\,$f^+(\Omega)=f^-(\Omega)$, and
$$ \int_{\mathcal{C}} L(\omega)\,\mathrm{d}Q(\omega) < +\infty,$$ \\
where \,$L(\omega)$\, is the length of the curve \,$\omega$, i.e. $L(\omega)=\int_0^1 ||\omega^\prime(t)||\,\mathrm{d}t$ (note that the length is measured according to the norm $||\cdot||$). We define the $\textit{traffic intensity}$\, $i_Q \in \mathcal{M}^+(\Omega)$ as follows
$$ \int_{\Omega} \phi \,\mathrm{d} i_Q = \int_{\mathcal{C}} \bigg(\int_0^1 \phi(\omega(t)) ||\omega^\prime(t)||\,\mathrm{d}t \bigg)\,\mathrm{d}Q(\omega) \,\,\,\,\,\,\mbox{for all}\,\,\,\phi \in C(\Omega).$$
\noindent This definition (which is taken from \cite{CarJimSan} and adapted to the case of a general norm) is a generalization of the notion of transport density $\sigma_\gamma$ \eqref{transport density def}. The interpretation is - again - the following: for a subregion $A$, $i_Q(A)$ represents the total cumulated traffic in $A$ induced by $Q$, i.e., for every path we compute ``how long" it stays in $A$, and then we average on paths. We also associate a vector measure \,$w_Q$\, (called $\textit{traffic flow}$) with any traffic plan \,$Q$\, via\\
$$\int_{\Omega} \xi \cdot \mathrm{d} w_Q = \int_{\mathcal{C}} \bigg(\int_0^1 \xi(\omega(t)) \cdot \omega^\prime(t)\,\mathrm{d}t \bigg)\,\mathrm{d}Q(\omega)\,\,\,\,\,\mbox{for all}\,\,\,\,\xi \in C(\Omega,\mathbb{R}^d).$$  
\\
Taking a gradient field \,$\xi=\nabla \phi$\, in the previous definition yields\\
\begin{equation*}
 \int_{\Omega} \nabla \phi \cdot \mathrm{d} w_Q = \int_{\mathcal{C}} \big( \phi(\omega(1)) - \phi(\omega(0)) \big)\,\mathrm{d}Q(\omega)=\int_{\Omega} \phi \,\mathrm{d}((e_1)_{\#} Q - (e_0)_{\#} Q),
 \end{equation*}\\
where $e_t$ is the evaluation map at time $t$, i.e. $e_t(\omega):=\omega(t)$, for all \,$\omega \in \mathcal{C}$, $t \in [0,1]$. From now on, we will restrict our attention to admissible traffic plans \,$Q$, i.e. traffic plans such that $(e_0)_{\#} Q = f^+$\, and \,$(e_1)_{\#} Q = f^-$, since, in this case, $w_Q$ will be an admissible flow in \eqref{Beckmann'}, i.e. one has
$$\nabla \cdot w_Q=f^+ - f^-.$$

\begin{lemma} \label{lemma flow decomposite}
Let \,$w$ be a flow such that \,$\nabla \cdot w = f^+ - f^-$. Then, there is an admissible traffic plan \,$Q$\, such that \,$||w-w_Q||(\Omega) +  i_Q(\Omega)= ||w||(\Omega)$.
\end{lemma}
\begin{proof}
The result is just a variant of what is presented in \cite[Section 4.2.3]{8}, and the proof will also follow the same lines. Following \cite[Section 4.2.3]{8}, we first assume the case where \,$f^+,\,f^-$ and \,$w$ are smooth with $f^+,f^- >0$, and obtain existence of an admissible traffic plan \,$Q$\, with \,$w_Q=w$\, and \,$i_Q=||w||$\, via a Dacorogna-Moser construction.

For the general case: following again \cite[Theorem 4.10]{8}, convolve $w$ (resp. $f^+$ and $f^-$) with a Gaussian kernel $\eta_\varepsilon$ and take care of the boundary behavior (for more details, see \cite[Lemma 4.8]{8}); we obtain smooth vector fields \,$w_{\varepsilon}$ and strictly positive smooth densities \,$f^\pm_{\varepsilon}$\,\, with \,$\nabla \cdot w_{\varepsilon}=f^+_{\varepsilon} - f^-_{\varepsilon}$\, such that \,$w_\varepsilon \deb w$\, (but also $||w_\varepsilon|| \deb ||w||$ because of standard properties of convolutions) and \,$f^\pm_\varepsilon \deb f^\pm$.
Let $(Q_\varepsilon)_{\varepsilon}$ be the sequence of traffic plans such that, for every $\varepsilon >0$, $w_{Q_\varepsilon}=w_\varepsilon$\, and \,$i_{Q_\varepsilon}=||w_\varepsilon||$. The measures \,$Q_\varepsilon$\, were constructed so that $(e_0)_{\#} Q_\varepsilon = f^+_\varepsilon$ and $(e_1)_{\#} Q_\varepsilon = f^-_\varepsilon$, which implies, at the limit, that \,$Q_{\varepsilon} \deb Q$\, (since $\int_{\mathcal{C}} L(\omega)\,\mathrm{d}Q_\varepsilon(\omega)=||w_\varepsilon||(\Omega) \leq C$ and so $Q_\varepsilon$ is tight) with $(e_0)_{\#}Q=f^+$ and $(e_1)_{\#}Q=f^-$. Moreover, it is not difficult to check that Proposition 4.7 in \cite{8} is still true if we replace the Euclidean norm $|\cdot|$ by an arbitrary norm $||\cdot||$. In particular, we have
\begin{eqnarray*}
 \int_{\Omega} \xi \cdot \mathrm{d}w &=& \lim_{\varepsilon} \int_{\Omega} \xi \cdot \mathrm{d}w_\varepsilon = \lim_\varepsilon  \int_{\mathcal{C}} \bigg(\int_0^1 \xi(\omega(t)) \cdot \omega^\prime(t)\,\mathrm{d}t \bigg)\,\mathrm{d}Q_\varepsilon(\omega) \\
 &\geq &   \int_{\Omega} \xi \cdot \mathrm{d}w_Q \,\,+\,\, ||\xi||_{\star,\infty} \,\,(i_Q(\Omega) - ||w||(\Omega)),
 \end{eqnarray*}
for all $\xi \in C(\Omega,\mathbb{R}^d)$. Hence, $||w-w_Q||(\Omega) + i_Q(\Omega) \leq ||w||(\Omega)$. Yet, the other inequality is always true since $||w_Q|| \leq i_Q$. Then, we get that $||w-w_Q||(\Omega)  + i_Q(\Omega)  = ||w||(\Omega)$. 
\end{proof}
\begin{proposition} \label{optimal flow comes from an optimal plan with general norm}
Let $w$ be an optimal flow for \eqref{Beckmann'}, then there is an optimal transport plan \,$\gamma$\, for \eqref{kant|x-y|} such that \,$w=w_\gamma$.
\end{proposition}
\begin{proof} From Lemma \ref{lemma flow decomposite}, there is an admissible traffic plan $Q$ such that $||w-w_Q||(\Omega)  + i_Q(\Omega) = ||w||(\Omega)$. The optimality of the flow $w$ and the fact that $||w_Q|| \leq i_Q$ imply that $w_Q=w$\, and $i_Q=||w||$. Hence,\\
$$
||w||(\Omega)
\!=\! i_{Q}(\Omega)
\!=\! \int_{\mathcal{C}} L(\omega)\,\mathrm{d}Q(\omega)
 \geq\!  \int_{\mathcal{C}} \!||\omega(0) \!-\! \omega(1)||\mathrm{d}Q (\omega) \!=\! \int_{\Omega \times \Omega}\!\!\! ||x-y||\,\mathrm{d}(e_0,e_1)_{\#}Q
  \geq \! \min\eqref{kant|x-y|}.$$\\
Yet, the equalities \,$||w||(\Omega)=\min\eqref{Beckmann'}=\min\eqref{kant|x-y|}$\, imply that the above inequalities are in fact equalities. This means that \,$Q$\, must be concentrated on segments (thanks to the strict convexity of the norm $|| \cdot ||$). Also, the measure $\gamma=(e_0,e_1)_{\#}Q$, which belongs to $\Pi(f^+,f^-)$, must be optimal in \eqref{kant|x-y|} and, we have \,$w=w_Q=w_{\gamma}$.\end{proof} 
We also introduce an easy stability result that we will need later on.
\begin{proposition}\label{stabsigma}
Suppose \,$f^+ \in \mathcal{M}^+(\Omega)$ is fixed and \,$f^-_n\deb f^-$. Let \,$\gamma_n$\, be an optimal transport plan between \,$f^+$ and $f^-_n$. Then, up to a subsequence, $\gamma_n\deb \gamma$, where \,$\gamma$\, is an optimal transport plan between \,$f^+$\, and \,$f^-$. Moreover,  if all the plans $\gamma_n$ are induced by transport maps $T_n$ and $\gamma$ is induced by a map $T$, then we have $T_n\to T$ in $L^2(f^+)$.
\end{proposition}
\begin{proof}
Firstly, we see easily that, up to a subsequence, $\gamma_n$ admits a weak limit $ \gamma$ in the sense of measures. The condition $\gamma_n \in \Pi(f^+,f^-_n)$ passes to the limit, thus giving $\gamma \in \Pi(f^+,f^-)$. Moreover, for each $n$, there is a corresponding Kantorovich potential \,$\phi_n$, which is 1-Lip according to $||\cdot||$, such that
$$\int_\Omega \phi_n \,\mathrm{d}(f^+ - f_n^-)=\int_{\Omega \times \Omega}||x-y||\,\mathrm{d}\gamma_n.$$
Up to a subsequence, we can suppose $\phi_n \rightarrow \phi$ uniformly in $\Omega$, where $\phi$ is also a 1-Lip function with respect to $||\cdot||$. Then, passing to the limit in the above equality, we get\\
$$\int_\Omega \phi \,\mathrm{d}(f^+ - f^-)=\int_{\Omega \times \Omega}||x-y||\,\mathrm{d}\gamma,$$
which is sufficient to infer that $\gamma$\, is actually an optimal transport plan between \,$f^+$ and \,$f^-$, and \,$\phi$\, is the corresponding Kantorovich potential. 

The last part of the statement, when plans are induced by maps, can be deduced by the weak convergence of the plans. Using \,$\gamma_n=(id,T_n)_\# f^+$\, and \,$\gamma_n\deb \gamma:=(id,T)_\# f^+$ and testing the weak convergence against the test function \,$\phi(x,y)=\xi(x)\cdot y$\, we obtain 
$$\int \xi(x)\cdot T_n(x)\,\dd f^+(x)\to \int \xi(x)\cdot T(x)\,\dd f^+(x),$$ 
which means that we have the weak convergence $T_n\deb T$ in $L^2(f^+)$. We can now test against $\phi(x,y)=|y|^2$ and obtain
$$\int |T_n(x)|^2\,\dd f^+(x)\to \int |T(x)|^2\,\dd f^+(x),$$\\
which proves the convergence of the $L^2$ norm. This gives strong convergence in $L^2(f^+)$.    
\end{proof}

 We have now to consider the case where the measures \,$f^+$ and \,$f^-$ are concentrated on the boundary.  As we said, the theory of existence of optimal maps, even for general norms, is now well-developed, but requires at least $f^+$ to be absolutely continuous. Hence, it  cannot be applied here. Moreover, uniqueness of the optimal map is in general not guaranteed. 

Surprisingly, the case of measures concentrated on the boundary of a strictly convex domain turns out to be easier. We can indeed prove in some cases that any optimal $\gamma$ in this case is induced by a transport map, which also implies uniqueness of $\gamma$ and of $\sigma_\gamma$. We first start from the case \,$d=2$ (with a general strictly convex norm $||\cdot||$) which is easier to deal with.

From now on we will suppose the condition that $f^+$ and $f^-$ have no common mass, which means that there exist two disjoint sets $A^+$ and $A^-$ contained in $\partial\Omega$ with $f^\pm$ concentrated on $A^\pm$ (beware that these sets are not necessarily the two supports of $f^+$ and $f^-$).

\begin{proposition} \label{Unicity}
Suppose that \,$\Omega$\, is strictly convex, and \,$d=2$. Then, if \,$f^+$ is atomless (i.e., $f^+(\{x\})=0$\, for every \,$x \in \partial\Omega$) and $f^+$ and $f^-$ have no common mass,  there is a unique optimal transport plan \,$\gamma$\, for (\ref{kant|x-y|}), between \,$f^+$ and \,$f^-$, and it is induced by a map \,$T$.
\end{proposition}
\begin{proof}
Let \,$\gamma$\, be an optimal transport plan between $f^+$ and $f^-$.
Let $\mathcal{D}$ be the set of double points, that is those points whose belong to several transport rays.
Take $x \in \mathcal{D}$ and let $r_x^\pm$ be two different transport rays starting from $x$. Let $\Delta_x \subset \Omega$ be the region delimited by $r_x^+$, $r_x^-$ and $\partial\Omega$. As $\Omega$ is strictly convex, 
then we see easily that $|\Delta_x|>0$ and the interior parts of all these sets $\Delta_x$, $x \in \mathcal{D}$, are disjoint (thanks also to the strict convexity of the norm $||\cdot||$).
This implies that the set \,$\mathcal{D}$\, is at most countable and so $f^+(\mathcal{D})=0$\, as \,$f^+$ is atomless. On the other hand, for every $x \in A^+\setminus \mathcal{D}$ there is a unique transport ray $r_x$ starting from $x$, and this ray $r_x$ intersects $A^-$ in - at  most - one point, which will be denoted by $T(x)$. Hence, we get that \,$\gamma=(id,T)_{\#} f^+$, which is equivalent to saying that $\gamma$\, is, in fact, induced by a map $T$. The uniqueness follows in the usual way: if two plans $\gamma$\, and \,$\gamma^\prime$ optimize (\ref{kant|x-y|}), the same should be true for $(\gamma + \gamma^\prime)/2$. Yet, for this measure to be induced by a map, it is necessary to have \,$\gamma=\gamma^\prime$.   
\end{proof}

The higher-dimensional counterpart of the above result should replace the assumption that $f^+$ is atomless with the assumption that $f^+$ gives no mass to $(d-2)$-dimensional sets (i.e. sets of codimension 1 within the boundary). Yet, this seems more complicated to prove, and we will just stick to an easier result, in the case where $f^+$ is absolutely continuous w.r.t. to the $\mathcal H^{d-1}$ measure on $\partial\Omega$ (that we simply write $f^+\in L^1(\partial\Omega)$). Unfortunately, the easy proof that we provide here below only works when the norm $||\cdot||$ is the Euclidean norm $|\cdot|$. 

\begin{proposition} \label{Unicity-d>2}
Suppose that \,$\Omega$\, is strictly convex and \,$d\geq 2$. Then, if \,$f^+\in L^1(\partial\Omega)$ and $f^+$ and $f^-$ have no common mass, there is a unique optimal transport plan \,$\gamma$\, for (\ref{kant|x-y|}) with the Euclidean cost $|x-y|$, and it is induced by a map \,$T$.
\end{proposition}
\begin{proof}
Let $\gamma$\, be an optimal transport plan between $f^+$ and $f^-$. According to the strategy above, it is enough to prove that for $f^+$-a.e. $x\in A^+$ there is at most a unique point $y\in A^-$ such that $(x,y)\in \spt\gamma$. We will parametrize $A^\pm$ via variables $s^\pm\in\R^{d-1}$. This is for sure possible since both $A^\pm$ do not fill the whole boundary $\partial\Omega$, and every proper subset of such a boundary is homeomorphic to a subset of $\R^{d-1}$, via an homeomorphism which can also be chosen to be locally bi-Lipschitz. Up to removing a negligible set, we can also assume that it is differentiable everywhere. Under this parameterization, we face a new transport problem in $\R^{d-1}$, with a new cost function $c(s^+,s^-):=|x(s^+)-y(s^-)|$, where $s^+\mapsto x(s^+)$ and $s^-\mapsto y(s^-)$ are the above parameterization of $A^+$ and $A^-$. 

Using standard arguments from optimal transport theory (see \cite[Chapter 1]{8}) one can see that the Kantorovich potentials in this new transport problem are locally Lipschitz continuous, and hence differentiable a.e. Thus it is enough to check that $c$ satisfies the twist condition to prove that $\gamma$ is necessarily induced by a map $T$, and that it is unique. Computing the gradient of $c$ w.r.t. the variable $s^+$ one gets
$$\nabla_{s^+} c(s^+,s^-)=\frac{x(s^+)-y(s^-)}{|x(s^+)-y(s^-)|}Dx(s^+),$$
where $Dx(s^+)$ is the Jacobian matrix of the diffeomorphism $x$. We need to prove that this expression is injective in $s^-$. Having two different values of $s^-$ (say, $s^-_0$ and $s^-_1$) where these expressions coincide means, using that  $s^+\mapsto x(s^+)$ is a diffeomorphism, that the two unit vectors $x(s^+)-y(s^-_i)/|x(s^+)-y(s^-_i)|$ have the same projection onto the tangent space to $\partial\Omega$ at $x(s^+)$ (note that, from $A^+\cap A^-=\emptyset$, we can assume $x(s^+)\neq y(s^-_i)$). Since they are unit vectors, and they both point to the interior of $\Omega$, which is convex, then they should fully coincide. But this means that the direction connecting $x(s^+)$ to the points $y(s^-_i)$ is the same, and since all these points lie on the boundary of a strictly convex domain, we have $y(s^-_0)=y(s^-_1)$. 
\end{proof}

\noindent For the sake of the next section, we also want stability results on the transport density. Suppose that $f^+$ and $f^-$ are fixed, and that a unique optimal transport plan $\gamma$ exists in the transportation from $f^+$ to $f^-$. In this case we will directly write \,$\sigma$\, instead of \,$\sigma_\gamma$, if no ambiguity arises. Given the optimal transport plan $\gamma$, let us define the measure \,$f_t$ via 
\begin{equation}\label{defft}
f_t=(\Pi_t)_{\#}(||x-y|| \cdot \gamma)
\end{equation}
where \,$\Pi_t(x,y):=(1-t)x + ty$. From (\ref{transport density def}), the transport density $\sigma$ may be easily written as
$$ \sigma=\int_0^1 f_t\,\mathrm{d}t.$$
We also define a sort of partial transport density that will be useful in the sequel: given $\tau\leq 1$, set
\begin{equation}\label{defsigmatau}
 \sigmatau=\int_0^\tau f_t\,\mathrm{d}t.
 \end{equation}
 Note that \,$\sigmatau$ really depends on $\gamma$, i.e., differently from $\sigma$, it is not in general true that different optimal plans $\gamma$ induce the same $\sigmatau$. On the other hand, we will only use this partial transport density in cases where the optimal $\gamma$ is unique. In this case we can also obtain:

\begin{proposition}\label{stabsigma-2}
Suppose \,$f^+ \in \mathcal{M}^+(\Omega)$ is fixed and \,$f^-_n\deb f^-$. Let \,$\gamma_n$\, be an optimal transport plan between \,$f^+$ and \,$f^-_n$ and suppose that there is a unique optimal transport plan between $f^+$ and $f^-$. Fix $\tau\leq 1$ and define $\sigmatau_n$ according to \eqref{defft} and \eqref{defsigmatau} using $\gamma_n$, and $\sigmatau$ using $\gamma$. Then, we have $\sigmatau_n\deb\sigmatau$.
\end{proposition}
\begin{proof}
This is a simple consequence of Proposition \ref{stabsigma}, of the continuity of the function $(x,y)\mapsto ||x-y||$, and of the uniqueness of the optimal $\gamma$.   
\end{proof}

\section{$L^p$ summability of boundary-to-boundary transport densities} \label{Sec.3}
In all that follows, $\Omega$ is a compact and uniformly convex domain in $\mathbb{R}^d$, $f^+$ and $f^-$ are two positive Borel measures concentrated on the boundary, and
at least one of them will belong to $L^1(\partial\Omega)$. Since we are only interested in the transport density between these two measures, we can always assume that they have no common mass, as the transport density only depends on the difference $f^+-f^-$ and common mass can be subtracted to both of them.

\noindent In the case \,$d=2$, by Proposition \ref{Unicity}, there will exist one unique optimal transport plan between these two
measures, while Proposition \ref{Unicity-d>2} gives the existence of a unique optimal transport plan between them in arbitrary dimension \,$d \geq 2$, but under the assumption that the norm \,$||\cdot||$ is the Euclidean one $|\cdot|$.
For simplicity of exposition, we say that the assumption {\bf(UA)} (standing for ``Uniqueness Assumption'') holds if
$${\bf{(UA)}}\qquad\mbox{either}\,\,\,d=2\,\,\,\,\mbox{or}\,\,\,\,d \geq 2 \,\,\,\,\mbox{and}\,\,\,||\cdot||\,\,\,\mbox{is the Euclidean norm}.$$
 We will make use of the transport density \,$\sigma$\, and of \,$\sigmatau$, defined in \eqref{defsigmatau} and provide estimate on them.
The main point is the following estimate.

\begin{proposition}\label{Lptau}
Suppose that the domain \,$\Omega \subset \mathbb{R}^d$ is uniformly convex, with all its curvatures bounded from below by a constant $\kappa>0$, take $p>1$ and $f^+\in L^p(\partial\Omega)$. If $p>2$ also suppose $\int f^+(x)^p\mathrm{d}(x,\spt(f^-))^{2-p}d\surf(x)<+\infty$. If \,{\bf{(UA)}} holds, then there exists a constant $C=C(\kappa,\diam(\Omega))$ such that we have

$$\int_\Omega |\sigmatau|^p\dd x\leq C\left(\int_0^\tau\frac{1}{(1-t)^{(d-1)(p-1)}}\dd t \right)\int_{\partial\Omega}f^+(x)^pD(x)^{2-p}\dd\surf(x),$$\\
where $D(x):=|x-T(x)|$ is the distance between each point $x\in\partial\Omega$ and its image $T(x)$ in the optimal transport map which induces the optimal plan $\gamma$. 
\end{proposition} 
\begin{proof} 
Following the same strategy as in \cite{7}, we first assume that the target measure $f^-$ is finitely atomic (the points $(x_j)_{j=1,...,m}$ being its atoms). Let \,$T$\, be the optimal transport map from $f^+$ onto $f^-$. For all $j \in \{1,...,m\}$, consider $T^{-1}(\{x_j\}) \subset \partial\Omega$, and partition it in finitely many smaller part, so that each can be represented by a single smooth chart parameterizing a part of $\partial\Omega$. We will call $(\chi_i)_{i=1,...,n}$ these parts. Let us call $\Omega_i$ the union of all transport rays starting from points in $\chi_i$, all these rays pointing to a common point $x_{j(i)}$ (but we will write $x_i$ for simplicity). Call $\Omega_i^{(\tau)}$ the set of points of the form $(1-t)x+tx_i$, with $x\in\chi_i$\, and \,$t\leq \tau$. The sets  $\Omega_i$ (and hence also $\Omega_i^{(\tau)}$) are essentially disjoint (the mutual intersections between them are Lebesgue-negligible). 

Set \,$\sigmatau_i:=\sigmatau \res \Omega_i$, for every $i \in \{1,...,n\}$. Of course, $\sigmatau_i$ is concentrated on $\Omega_i^{(\tau)}$.  In order to get \,$L^p$\, estimates on $\sigmatau$, we want to give an explicit formula of each \,$\sigmatau_i$. Fix $i \in \{1,...,n\}$ and let $\alpha_i$ be a regular function such that, up to choosing a suitable system of coordinates, $\chi_i$ is contained in the graph of $s \mapsto \alpha_i(s)$, with $s\in\tilde\chi_i\subset\R^{d-1}$ (hence, the sets $\tilde\chi_i$ are the $(d-1)$-dimensional domains where the charts are defined). For every \,$y \in \Omega_i^{(\tau)}$, there are a unique point $x=(s,\alpha_i(s)) \in \chi_i$\, and \,$t \in [0,\tau]$ such that
$$y:=(y',y_d)=(1-t)x + t x_i=((1-t)s + t x_{i}',(1-t) \alpha_i(s)+t x_{i,d}),$$
where we write $x_i:=(x_{i}',x_{i,d})$ by separating the last (vertical) coordinate from the others.
For all \,$\varphi \in C(\Omega_i)$, we get 
\begin{eqnarray*}
\int_{\Omega_i} \varphi(y) \,\mathrm{d}\sigmatau_i(y)
 &=& \int_{\chi_i} \int_0^{\tau} \varphi((1-t)x+tx_i)||x-x_i||\,\mathrm{d}t\,\mathrm{d}f^+(x)\\
 &=& \int_{\Omega_i^{(\tau)}} \varphi(y)\,\frac{||(s,\alpha_i(s))-x_i||\,f^+(s,\alpha_i(s))\,\sqrt{1+|\nabla{\alpha_i}(s)|^2}}{J_i(t,s)}\,\mathrm{d}y,
\end{eqnarray*} 
where 
$J_i(t,s):=|\det(D_{(s,t)}(y',y_d))|$.
Hence, we get
\begin{eqnarray} \label{our transport density}
\sigmatau(y) &=& \frac{||(s,\alpha_i(s))-x_i||\,f^+(s,\alpha_i(s))\,\sqrt{1+|\nabla{\alpha_i}(s)|^2}}{J_i(t,s)},\,\,\,\,\,\mbox{for a.e.}\,\,\,y \in \Omega_i^{(\tau)}.
\end{eqnarray}
We then have
\begin{eqnarray*} \label{estimates}
||\sigmatau||_{L^p(\Omega)}^p &=& \sum_{i=1}^n\int_{\tilde\chi_i} \int_0^{\tau} \sigmatau((1-t)x + t x_i)^p J_i(t,s)\,\mathrm{d}t\,\mathrm{d}s \\  
&=& \sum_{i=1}^n \int_{\chi_i} \int_0^{\tau} \frac{||x-x_i||^p{f^+(x)}^p (1+|\nabla{\alpha_i}(s)|^2)^{\frac{p-1}{2}}}{J_i(t,s)^{p-1}}\,\mathrm{d}t\,\mathrm{d}\surf (x).
\end{eqnarray*}
Compute
$$D_{(s,t)}(y',y_d)=\begin{pmatrix}
 (1-t)\mathrm{I}\,&\,\, x_i'-s \\
 (1-t)\nabla{\alpha_i}(s) & x_{i,d} - \alpha_i(s) 
 \\ 
 \end{pmatrix},$$
 where \,$\mathrm{I}$\, is the $(d-1)\times(d-1)$ identity matrix. Up to considering sets $\chi_i$ which are very small, each one close to a point $x\in \partial\Omega$, and choosing a coordinate system where the vertical coordinate is parallel to the normal vector to $\partial \Omega$ at $x$, we can assume that $\nabla{\alpha_i}(s)$ is very small. At the limit, we can compute the above determinant as if it vanished, and thus we get $J_i(t,s)=(1-t)^{d-1}(x_{i,d}-\alpha(s))$ (as well as $1+|\nabla{\alpha_i}(s)|^2=1$). This allows to write the change-of-variable coefficients in an intrinsic way, and thus obtain
 \begin{equation*} \label{estimates2}
||\sigmatau||_{L^p(\Omega)}^p =\sum_{i=1}^n \int_{\chi_i} \int_0^{\tau} \frac{||x-x_i||^p{f^+(x)}^p }{(1-t)^{(d-1)(p-1)}\big((x_i-x)\cdot \mathbf{n}(x)\big)^{p-1}}\,\mathrm{d}t\,\dd\surf(x),
\end{equation*}
where $\mathbf{n}(x)$ is the inward normal vector to $\partial\Omega$ at $x$. Using the lower bound on the curvature of $\partial\Omega$ we have, for every pair of points \,$x$\, and \,$x_i$\, on $\partial\Omega$: 
$$(x_i-x)\cdot \mathbf{n}(x)\geq c|x-x_i|^2,$$
for a constant $c=c(\kappa,\diam\Omega)$. Using then \,$x_i=T(x)$\, for \,$x\in\chi_i$\, and the equivalence between the two norms $||\cdot||$ and $|\cdot|$, this provides the desired formula
 \begin{equation*} \label{estimates2}
||\sigmatau||_{L^p(\Omega)}^p \leq C\int_{\partial\Omega} \int_0^{\tau} \frac{|x-T(x)|^{2-p}{f^+(x)}^p }{(1-t)^{(d-1)(p-1)}}\,\mathrm{d}t\,\dd\surf(x).
\end{equation*}
%
This proves the claim when $f^-$ is atomic. If not, take a sequence $(f_n^-)_n$ of atomic measures converging to $f^-$ and concentrated on $\spt(f^-)$. Call $T_n$ the optimal maps from $f^+$ to $f^-_n$ and $D_n(x)=|x-T_n(x)|$. By Proposition \ref{stabsigma-2}, the partial transport densities \,$\sigmatau_n$ converge to the corresponding partial transport density \,$\sigmatau$ and by Proposition \ref{stabsigma} the optimal transport maps $T_n$ also converge a.e. to the optimal transport map $T$ inducing $\gamma$ (up to extracting a subsequence, since $L^2(f^+)$ convergence implies a.e. convergence up to a subsequence). Moreover,  we have $ \int_{\partial\Omega}D_n(x)^{2-p}{f^+(x)}^p \,\dd\surf(x)\to  \int_{\partial\Omega}D(x)^{2-p}{f^+(x)}^p \,\dd\surf(x)$ by dominated convergence, using either $p\leq 2$\, and $f^+\in L^p(\partial\Omega)$ or $ \int_{\partial\Omega}\mathrm{d}(x,\spt(f^-))^{2-p}{f^+(x)}^p \,\dd\surf(x)<+\infty$, according to our assumptions. Using semicontinuity on the left hand side, we get 
$$||\sigmatau||_{L^p(\Omega)}^p \leq \liminf_n ||\sigmatau_n||_{L^p(\Omega)}^p \leq C\,\bigg(\int_0^\tau\frac{1}{(1-t)^{(d-1)(p-1)}}\,\mathrm{d}t \bigg) \int_{\partial\Omega}D(x)^{2-p}{f^+(x)}^p \,\dd\surf(x)$$
and the result is proven in general.  
\end{proof}

From the above estimate, we can deduce many integrability results.
\begin{proposition}\label{Lpd-1}
Let $\Omega$ be a uniformly convex domain in $\mathbb{R}^d$ and suppose $f^+\in L^p(\partial\Omega)$ with $p<d/(d-1)$. If \,{\bf{(UA)}} holds, then the transport density $\sigma$ between $f^+$ and any $f^-\in \mathcal{M}^+(\partial\Omega)$ is in $L^p(\Omega)$. \end{proposition}
\begin{proof}
Note that our assumption on $p$ implies $p\leq 2$. To prove this result it is enough to use Proposition \ref{Lptau} with $\tau=1$, since in this case the integral $\int_0^1\frac{1}{(1-t)^{(d-1)(p-1)}}\dd t$ converges, and the term $D(x)^{2-p}$ is bounded since $p\leq 2$. 
\end{proof}

\begin{proposition}\label{Lp2}
Let \,$\Omega$\, be a uniformly convex domain in \,$\mathbb{R}^d$ and suppose that \,$f^+,f^-\in L^p(\partial\Omega)$ with \,$p\leq 2$. If \,{\bf{(UA)}} holds, then the transport density \,$\sigma$ between these two measures is in $L^p(\Omega)$.\end{proposition}
\begin{proof}
In this case the integral in the estimate of $\sigma=\sigmatau$ with $\tau=1$ can diverge, so we need to adapt our strategy. Following again \cite{7}, we write
$$\sigma=\sigma^+ + \sigma^-,$$
where \,$\sigma^+=\sigma^{(1/2)}$\, and \,$\sigma^-=\sigma-\sigma^{(1/2)}$. In this case the $L^p$ summability of $f^+$ guarantees that of $\sigma^+ $ since $p\leq 2$ implies that $D(x)^{2-p}$ is bounded. Symmetrically, the $L^p$ summability of $f^-$ guarantees that of $\sigma^- $.

Note that, thanks to Propositions \ref{Unicity} and \ref{Unicity-d>2}, we  do not face the same difficulties as in \cite{7}, where it was not obvious to glue together estimates on $\sigma^+$ obtained by approximating $f^-$ and estimates on $\sigma^-$ coming from the approximation of $f^+$.   
\end{proof}

We will see in Section \ref{Sec.4} that the same result is false for $p>2$, and that in order to obtain higher integrability we need to assume much more on $f^+$ and $f^-$.

\begin{remark} We do not discuss it here in details, but the summability result also works for Orlicz spaces with growth less than quadratic, i.e. we have, for every convex and superlinear function $\Psi=\R_+\to\R_+$ with $\Psi(s)\leq C(s^2+1)$,
$$\int_\Omega \Psi(\sigma(x))\dd x\leq C \int_{\partial\Omega}\Psi(|f(x)|)\dd\surf(x)+C.$$
This can be proven in similar ways with suitable manipulations on the function $\Psi$. In particular, this implies that $f\in L^1(\partial\Omega)\impl \sigma\in L^1(\Omega)$.
\end{remark}

\begin{proposition}\label{LpCalpha}
Suppose that \,$f^+,f^-\in C^{0,\alpha}(\partial\Omega)$\, for \,$0<\alpha\leq 1$. If \,{\bf{(UA)}} holds, then the transport density $\sigma$ between these two measures is in $L^p(\Omega)$\, for \,$p=2/(1-\alpha)$ (with $p=\infty$\, for \,$\alpha=1$).
\end{proposition}
\begin{proof}
First, we check that we can apply Proposition \ref{Lptau}, since in this case we need to use $p=2/(1-\alpha)>2$. Consider a point $x$ with $f^+(x)>0$, and take a point $y\in \spt(f^-)$ with $|x-y|=\mathrm{d}(x,\spt(f^-))$. Then we have $f^+(y)=0$ (since $f^+$ and $f^-$ have no mass in common) and $f^+(x)=|f^+(x)-f^+(y)|\leq C|x-y|^\alpha$. This provides $\mathrm{d}(x,\spt(f^-))^{2-p}{f^+(x)}^p\leq C\mathrm{d}(x,\spt(f^-))^{2-p+p\alpha}$. With our choice of $p$, this quantity is bounded since the exponent is non-negative (for $\alpha<1$ the choice $p= 2/(1-\alpha)$ provides a zero exponent; for $\alpha=1$ this exponent is equal to \,$2$\, for any \,$p$). This in particular guarantees  $ \int_{\partial\Omega}\mathrm{d}(x,\spt(f^-))^{2-p}{f^+(x)}^p \,\dd\surf(x)<+\infty$. Of course, the same can be performed on $f^-$. Then, the same strategy as in Proposition \ref{Lp2} shows
$$||\sigma||_{L^p(\Omega)}^p \leq C \bigg(\int_{\partial\Omega}D(x)^{2-p}{f^+(x)}^p \,\dd\surf(x)+\int_{\partial\Omega}D^-(x)^{2-p}{f^-(x)}^p \,\dd\surf(x)\bigg),$$
where \,$D^-(x):=|x-T^{-1}(x)|$\, is defined as $D(x)$, but relatively to $f^-$. Using $D(x)=|x-T(x)|\geq \mathrm{d}(x,\spt(f^-))$ and the fact that the exponent $2-p$ is negative, the quantity $|D(x)|^{2-p}{f^+(x)}^p\leq CD(x)^{2-p+p\alpha}$ is bounded. Since a similar argument can be performed on $f^-$, we obtain finiteness of the norm $||\sigma||_{L^p(\Omega)}$ (and for $\alpha=1$ we obtain $\sigma\in L^\infty$ by passing to the limit $p\to\infty$). \end{proof}
\section{Counter-example to the $L^{2+\varepsilon}$ summability} \label{Sec.4}
 In this section, we show that the $L^p$ estimates for the transport density, in the case where $p>2$, fail even if we assume $f^\pm\in L^\infty(\partial\Omega)$. More precisely, we will construct an example of $f^\pm$, where $f^\pm \in L^\infty(\partial\Omega)$,
 but the transport density $\sigma$ between them does not belong to $L^{2+\varepsilon}(\Omega)$ for any $\varepsilon>0$. For simplicity, this will be done in dimension $d=2$ and the norm $||\cdot||$ will be the Euclidean one.
 
  Let $\Omega$ be a disk and let $(\chi_n^\pm)_n$ be a sequence of arcs in $\partial\Omega$ such that $\mathcal{H}^1(\chi_n^\pm)=\ve_n$, for some sequence $\ve_n$ to be chosen later. We will put these arcs one after the other, so that they only have endpoints in common, and we assume that they are ordered in the following way: $\chi_{n-1}^+,\,\chi_{n-1}^-,\,\chi_{n}^-,\,\chi_{n}^+,\,\chi_{n+1}^+,\,\chi_{n+1}^-$, for all $n$ (see Figure \ref{Fig}).
 \begin{figure}[h]
 \begin{tikzpicture}[scale=0.8,line cap=round,line join=round,>=triangle 45,x=1cm,y=1cm]
\clip(-3.2,-3.3) rectangle (3.6,3.6);
\draw(0,0) circle (3cm);
\draw [shift={(0,0)},color=ffqqtt]  plot[domain=-4.71:0.91,variable=\t]({1*3*cos(\t r)+0*3*sin(\t r)},{0*3*cos(\t r)+1*3*sin(\t r)});
\draw [shift={(0,0)},color=ffqqtt]  plot[domain=0.26:6.06,variable=\t]({1*3*cos(\t r)+0*3*sin(\t r)},{0*3*cos(\t r)+1*3*sin(\t r)});
\draw [shift={(0,0)},color=ttttff]  plot[domain=0.91:1.57,variable=\t]({1*3*cos(\t r)+0*3*sin(\t r)},{0*3*cos(\t r)+1*3*sin(\t r)});
\draw [shift={(0,0)},color=ttttff]  plot[domain=0.26:0.91,variable=\t]({1*3*cos(\t r)+0*3*sin(\t r)},{0*3*cos(\t r)+1*3*sin(\t r)});
\draw [shift={(0,0)},color=qqffqq]  plot[domain=-0.22:0.26,variable=\t]({1*3*cos(\t r)+0*3*sin(\t r)},{0*3*cos(\t r)+1*3*sin(\t r)});
\draw [shift={(0,0)},color=qqffqq]  plot[domain=5.66:6.06,variable=\t]({1*3*cos(\t r)+0*3*sin(\t r)},{0*3*cos(\t r)+1*3*sin(\t r)});
\draw [shift={(0,0)},color=ffffqq]  plot[domain=5.1:5.39,variable=\t]({1*3*cos(\t r)+0*3*sin(\t r)},{0*3*cos(\t r)+1*3*sin(\t r)});
\draw [shift={(0,0)},color=ffffqq]  plot[domain=5.39:5.66,variable=\t]({1*3*cos(\t r)+0*3*sin(\t r)},{0*3*cos(\t r)+1*3*sin(\t r)});
\draw (0,3)-- (2.9,0.78);
\draw (2.9,0.78)-- (2.44,-1.74);
\draw (2.44,-1.74)-- (1.14,-2.78);
\draw (0.56,2.95)-- (2.7,1.31);
\draw (1.18,2.76)-- (2.38,1.83);
\draw (2.99,0.22)-- (2.7,-1.31);
\draw (2.92,0.67)-- (2.52,-1.6);
\begin{scriptsize}
\draw[color=black] (1.12,3.2) node {$\chi_{n-1}^+$};
\draw[color=black] (3,1.9) node {$\chi_{n-1}^-$};
\draw[color=black] (3.4,0.34) node {$\chi_n^-$};
\draw[color=black] (3.1,-1.4) node {$\chi_n^+$};
\draw[color=black] (2.6,-2.25) node {$\chi_{n+1}^+$};
\draw[color=black] (1.8,-2.8) node {$\chi_{n+1}^-$};
\draw[color=black] (2.4,-0.26) node {$\Delta_n$};
\end{scriptsize}
\end{tikzpicture}
\caption{\label{Fig}}
\end{figure}

  \noindent Set $f^\pm=\ind_{\chi^\pm}$, where $\chi^\pm=\cup_n \chi_n^\pm$, and let $T$ be the optimal transport map from $f^+$ to $f^-$. Correspondingly, let $\sigma$ be the transport density. We see easily that the restriction of \,$T$ to $\chi_n^+$ is the optimal transport map \,$T_n$ between $f_n^+$ and $f_n^-$, where $f_n^\pm$ is the restriction of $f^\pm$ to $\chi_n^\pm$. Moreover, if we denote by $\Delta_n$ the union of all transport rays from $f_n^+$ onto $f_n^-$, then the restriction of the transport density \,$\sigma$ to $\Delta_n$ is the transport density $\sigma_n$ between $f_n^+$ and $f_n^-$. We want to compute this density $\sigma_n$. Let $s \mapsto \alpha_n(s)$ be a parameterization of \,$\chi_n^+ \cup \chi_n^-$ where $s=0$ corresponds to the boundary point between the two arcs. It is clear that $T_n(s,\alpha_n(s))=(-s,\alpha_n(s))$, for all $s \in [0,\ve_n]$. Then, for every $y \in \Delta_n$, there is a unique $(t,s) \in [0,1] \times [0,\ve_n]$ such that
 $$y:=(y_1,y_2)=(1-t)(s,\alpha_n(s)) + t(-s,\alpha_n(s))=((1-2t)s,\alpha_n(s)).$$\\
 Hence, for every $\varphi \in C(\Delta_n)$, we have
\begin{eqnarray*}
\int_{\Delta_n} \varphi(y)\,\sigma_n(y)\,\mathrm{d}y &=&\int_{\chi_n^+} \int_0^1 \varphi((1-t)x+tT_n(x)) |x-T_n(x)| f_n^+(x) \,\mathrm{d}t\,\mathrm{d}x\\ 
&=&\int_{0}^{\ve_n} \int_0^1 \varphi((1-2t)s,\alpha_n(s))\,2s\, \sqrt{1+{\alpha_n^\prime(s)}^2} \,\mathrm{d}t\,\mathrm{d}s\\ 
&=& \int_{\Delta_n} \varphi(y)\,\frac{2s\, \sqrt{1+{\alpha_n^\prime(s)}^2}}{J_n(t,s)} \,\mathrm{d}y,
\end{eqnarray*}\\ \\
where \,$J_n(t,s):=|\det(D_{(t,s)}(y_1,y_2))|$\, on \,$\Delta_n$.
This provides
 \begin{equation*}
 \sigma_n(y) =\frac{2 s \sqrt{1+{\alpha_n^\prime(s)}^2}}{J_n(t,s)},\,\,\,\,\mbox{for a.e.}\,\,\,y \in \Delta_n.
 \end{equation*}
 Consequently, we obtain
 
 \begin{eqnarray*} \label{estimates}
||\sigma||_{L^p(\Omega)}^p &=& \sum_{n=1}^\infty\int_0^{\ve_n} \int_0^{1} \sigma_n((1-2t)s,\alpha_n(s))^p J_n(t,s)\,\mathrm{d}t\,\mathrm{d}s \\  
&\approx & \sum_{n=1}^\infty \int_{0}^{\ve_n} \int_0^{1} \frac{s^p}{J_n(t,s)^{p-1}}\,\mathrm{d}t\,\mathrm{d}s.
\end{eqnarray*}
Computing
$$D_{(t,s)}(y_1,y_2)=\begin{pmatrix}
  - 2s \,\,&\,\, 1-2t \\
 0 \,\,&\,\, {\alpha_n^\prime(s)}
 \\
 \end{pmatrix},$$
 we get
 $$J_n(t,s)=2s\,{\alpha_n^\prime(s)} \approx s^2.$$
 Finally, we have
 \begin{eqnarray*} \label{estimate}
||\sigma||_{L^p(\Omega)}^p \approx \sum_{n=1}^\infty  \ve_n^{3-p}.
\end{eqnarray*}
This immediately shows that with this construction we cannot have $\sigma\in L^3$. Moreover, it is enough to choose a sequence $\ve_n$ satisfying
$$\sum_{n=1}^\infty  \ve_n<+\infty,\quad \sum_{n=1}^\infty  \ve_n^\beta=+\infty$$
for all $\beta<1$, to prove $\sigma\notin L^p(\Omega)$ for all \,$p>2$. Take for instance \,$\ve_n=\frac{1}{n(\log(1+n))^2}$.

\section{Applications to the BV least gradient problem}\label{Sec.5}

We collect in this section some corollaries of the results of the previous sections, which give interesting proofs for some properties of the BV least gradient problem in dimension $d=2$. We need to restrict to $d=2$ because only in this framework rotated gradients have prescribed divergence. As we saw, in dimension $d=2$ all our results are valid for arbitrary strictly convex norms, and hence apply to the anisotropic least gradient problem.

In all the cases, we will suppose $g\in BV(\partial\Omega)$. Note that this assumption is required to apply the classical theory of optimal transport to $f=\partial_{\mathbf{t}} g$; this requires to transport a measure onto another. If $g$ was only in $L^1(\partial\Omega)$, then $f$ would be the (one-dimensional) derivative of an $L^1$ function, i.e. an element of the dual of Lipschitz functions (since $W^{-1,1}=(W^{1,\infty})'$). It is not surprising that a Monge-Kantorovich theory is also possible in this case (because formula \eqref{dual} characterizes the transport cost as the dual norm to the Lipschitz norm), see \cite{BouChaJim}, but no estimates are possible.

\begin{proposition}
If \,$\Omega\subset \R^2$ is strictly convex and $g\in BV(\partial\Omega)$, then Problem \eqref{leastgradient-inf} has a solution (i.e. Problems \eqref{leastgradient-pen} and \eqref{leastgradient-ext} have a solution whose trace is $g$).
\end{proposition}
\begin{proof}
We have already discussed the fact that we just need to exclude that the solution of \eqref{leastgradient-pen} or \eqref{leastgradient-ext} has a part of its distributional derivative on the boundary. After the rotation, this means that its trace agrees with $g$ if and only if $\sigma(\partial\Omega)=0$. Yet, in strictly convex domains, the transport density does not give mass to the boundary, because of the representation formula \eqref{transport density def}. \end{proof}

\begin{proposition}\label{uniqBV}
If $\Omega\subset \R^2$ is strictly convex, and $g\in (BV\cap C^0)(\partial\Omega)$, then Problem \eqref{leastgradient-inf} has a unique solution.
\end{proposition}
\begin{proof}
Using again the rotation trick, we just need to prove uniqueness of the transport density. The condition $g\in C^0$ implies that its tangential derivative has no atoms, and we can apply Proposition \ref{Unicity}.   
\end{proof}

The following result is probably the main contribution of this paper to the understanding of the anisotropic least gradient problem, as we are not aware of similar results already existing in the literature. 
\begin{theorem}\label{onlynew}
If \,$\Omega\subset \R^2$ is uniformly convex, and \,$g\in W^{1,p}(\partial\Omega)$ with \,$p\leq 2$, then the unique solution of Problem \eqref{leastgradient-inf} belongs to $W^{1,p}(\Omega)$.
\end{theorem}
\begin{proof}
Setting \,$f=\partial_{\mathbf{t}} g$ and using $f^+$ and $f^-$ as its positive and negative parts, the condition $g\in  W^{1,p}(\partial\Omega)$  implies $f^\pm\in L^p(\partial\Omega)$. Hence, Proposition \ref{Lp2} implies $\sigma\in L^p(\Omega)$, and then $\nabla u\in L^p(\Omega,\mathbb{R}^2)$.
\end{proof}

\begin{proposition}
Even if \,$\Omega\subset \R^2$ is a disk, for every \,$p>2$ there exists \,$g\in \Lip(\partial\Omega)$ such that the unique $u$ solution of Problem \eqref{leastgradient-inf} is not in $W^{1,p}(\Omega)$.
\end{proposition}
\begin{proof}
It is enough to take \,$g$\, as the antiderivative of the function $f=f^+-f^-$ of the counter-example of Section \ref{Sec.4}.   
\end{proof}

\begin{proposition}\label{C1a}
If \,$\Omega\subset \R^2$ is uniformly convex, and \,$g\in C^{1,\alpha}(\partial\Omega)$ with $\alpha<1$, then the unique solution of Problem \eqref{leastgradient-inf} belongs to $W^{1,p}(\Omega)$ for \,$p=2/(1-\alpha)$.\end{proposition}
\begin{proof}
This is a consequence of Proposition \ref{LpCalpha}. \end{proof}

\begin{remark} Note that the above \,$W^{1,p}$ regularity also implies H\"older bounds, since in dimension $d=2$ we have $W^{1,p}\subset C^{0,1-2/p}$. In particular, using $p=2/(1-\alpha)$, we get $g\in C^{1,\alpha}(\partial\Omega)\impl u\in C^{0,\alpha}(\Omega)$. Yet, this bound is not optimal, as it is known (see, for instance, \cite{44}) that we have $g\in C^{1,\alpha}(\partial\Omega)\impl u\in C^{0,(\alpha+1)/2}(\Omega)$. It is interesting to note that one would obtain exactly the desired $C^{0,(\alpha+1)/2}$ behavior if it was possible to use the Sobolev injection of $W^{1,p}$ corresponding to dimension 1 instead of dimension 2. This seems reasonable, using the fact that level lines of $u$ are transport rays in the transport problem from $f^+$ to $f^-$, hence are line segments, but it is not easy to justify and goes beyond the scopes of this paper.
\end{remark}

\begin{proposition}\label{C11}
If \,$\Omega\subset \R^2$ is uniformly convex, and \,$g\in C^{1,1}(\partial\Omega)$, then the unique solution of Problem \eqref{leastgradient-inf} is Lipschitz continuous.
\end{proposition}
\begin{proof}
This is also a consequence of Proposition \ref{LpCalpha}.   
\end{proof}

\begin{remark}Note that the above Lipschitz result is optimal, and perfectly coherent with the theory involving the bounded slope condition (see, for instance \cite{Stampacchia,Cellina2001}), since $C^{1,1}$ functions on the boundary of uniformly convex domains satisfy the bounded slope condition (see \cite{Hartman}).
\end{remark}

We finish this section with two last remarks.

\begin{remark}
The strict convexity assumption on $\varphi$ is crucial in this framework in order to obtain the uniqueness of the minimizers (Proposition \ref{uniqBV}) and also for the approximation procedures performed in Section \ref{Sec.3} which allow to translate the estimates in the atomic case into estimates which are valid in the general case. Yet, the $L^p$ bounds obtained in Section \ref{Sec.3} do not depend on how much the norm $||\cdot||$ is strictly convex, and, whenever $c_0|\cdot|\leq ||\cdot ||\leq c_1|\cdot|$, the constants in the estimates only depend on $c_0$ and $c_1$. Hence, given an arbitrary norm $\varphi$ (which is of course equivalent to the Euclidean one), it is possible, just by approximating it with $\varphi_\ve(z)=\varphi(z)+\ve|z|$, to obtain the results of Theorem \ref{onlynew}, Proposition \ref{C1a} and Proposition \ref{C11}, but in this case the estimates will be true for at least one minimizer (the one selected by this approximation), as there is no more any guarantee of uniqueness.
\end{remark}
\begin{remark} We observe that we have not used Proposition \ref{Lpd-1} in this Section. Indeed, in the framework of the least gradient problem assuming assumptions on $f^+$ (i.e. on the positive part of the tangential derivative of the boundary datum) are not natural at all. Proposition \ref{Lpd-1} has been inserted in Section \ref{Sec.3} just because it was an easy consequence of Proposition \ref{Lptau}. Also consider that a simple result which could have been proven in Section \ref{Sec.3} was the implication $f^\pm\in L^p(\partial\Omega)\impl \sigma\in L^p(\Omega)$ for arbitrary $p$ (including $p>2$) under the assumption $\spt(f^+)\cap \spt(f^-)=\emptyset$, but we did not considered it because this assumption, in terms of $g$, is very innatural: it would mean that $g$ has some flat regions separating those with positive and negative derivatives.
\end{remark}

\end{document}